\documentclass{amsart}

\usepackage{ifpdf}

\allowdisplaybreaks[1]

\ifpdf
\usepackage{hyperref}           
\else
\usepackage[hypertex]{hyperref} 
\fi
\newtheorem{definition}{Definition}[section]
\newtheorem{theorem}[definition]{Theorem}
\newtheorem{lemma}[definition]{Lemma}
\newtheorem{corollary}[definition]{Corollary}

\theoremstyle{definition}
\newtheorem{remark}[definition]{Remark}

\newcommand\CC{\mathbb{C}}
\newcommand\NN{\mathbb{N}}
\newcommand\RR{\mathbb{R}}
\newcommand\TT{\mathbb{T}}
\newcommand\ZZ{\mathbb{Z}}

\newcommand\ds{\displaystyle}

\newcommand\style{\mathcal }          

\newcommand\hil{\style H}                        
\newcommand\h{\hil}
\newcommand\bh{\style B(\style H)}    
\newcommand\tr{ \mbox{Tr} } 

\begin{document}

\numberwithin{equation}{section}

\title{Majorisation and Kadison's Carpenter's Theorem}

\author{Mart\'\i n~Argerami}

\address{Department of Mathematics and Statistics,
University of Regina,
Regina, Saskatchewan S4S 0A2,
Canada}
\email{argerami@math.uregina.ca}

\thanks{This work is supported in part by the NSERC Discovery Grant program}

\begin{abstract}
We discuss Kadison's Carpenter's Theorems in the context of their relation to majorisation, and we offer a new proof of his striking characterisation of the set of diagonals of orthogonal projections on Hilbert space.
\end{abstract}

\maketitle


\section{Introduction}

Majorisation is a basic notion in matrix analysis that roughly measures the spread of the entries in vectors with the same ``weight''. The breadth of applications of majorisation is surprising---we refer the reader to \cite{MOA} for hundreds of pages of examples, and to section \ref{section: majorisation} for basic facts.

In an infinite-dimensional setting, majorisation has been considered by a variety of authors in many different contexts. Significant work was done on majorisation in von Neumann algebras by several authors in the 1980s, in particular F. Hiai \cite{Hiai1987,Hiai1989,HiaiNakamura1987}. Hiai proved several results analogous to Theorem \ref{theorem: equivalent conditions for majorisation} below, framed in the context of semifinite von Neumann algebras.
More recently, interest in infinite-dimensional majorisation was revived by the work of Neumann \cite{Neumann1999}, who proved a Schur-Horn theorem for diagonal operators in $\bh$.

Kadison's remarkable results in \cite{kadison2002a,kadison2002b} were motivated by
his infinite-dimensional generalisations of the Pythagorean Theorem and its converse
(named by him the Carpenter's Theorem). One can see his result as a characterisation of the diagonal of a projection. This point of view was strengthened in
\cite{arveson-kadison2006} where the Schur-Horn Theorem was considered from the
perspective of characterizing diagonals of selfadjoint operators. Indeed, one can see the
Schur-Horn Theorem \ref{theorem: Schur-Horn} as saying that diagonals of selfadjoint
operators characterize those vectors majorised by a fixed one; alternatively, one can say
that majorisation characterizes the diagonals of selfadjoint operators. When moving to
infinite-dimensional settings, the difference in point of view becomes significant: Neumann's
Theorem \cite[Corollary 2.18 and Theorem 3.13]{Neumann1999} can be seen as a satisfactory
characterisation of majorisation, but it is not satisfactory as a characterisation of
diagonals, as shown by Theorem \ref{theorem: teorema 15} below. A complete characterisation
of the diagonals of a selfadjoint operator in $\bh$ is an open problem at the moment of
writing, as it is the finite-dimensional case for normal matrices. Recent advances have been a full Schur-Horn Theorem in II$_1$-factors by
Ravichandran \cite{Ravichandran2012}---improving over versions with closure a-la-Neumann
in \cite{ArgeramiMassey2007,ArgeramiMassey2008a,ArgeramiMassey2013}---and a
generalisation of Kadison's Carpenter's Theorem (\cite[Theorem 15]{kadison2002b}, or Theorem
\ref{theorem: teorema 15} below) to selfadjoint operators with finite spectrum by Bownik
and Jasper \cite{BownikJasper2013b,Jasper2013}. We should also mention here recent work by Kaftal and
Weiss \cite{KaftalWeiss2008,KaftalWeiss2010} on versions of the Schur-Horn Theorem for
compact operators. Other infinite-dimensional versions of majorisation are linked to
generalisations related to the Horn conjecture \cite{Klyachko1998,Tao1999,Tao2004} and
their generalisation to von Neumann algebras
\cite{Bercovici2010,Bercovici2006,Bercovici2009,Collins2008,Collins2009,Collins2011} and,
among others, the topic of frames \cite{AntezanaMasseyRuizStojanoff2009}.

The goal of this article is three-fold: first we want to formulate Kadison's results, in a self-contained way, in the corresponding context of infinite-dimensional majorisation where we feel they belong; second, we want provide an alternative proof to his results in \cite{kadison2002b} that we hope will make them easier to follow and open to a broader audience; and third, we want to advertise the beauty of the mathematics and the connections that arise from \cite{kadison2002a,kadison2002b}.

\section{Preliminaries}\label{section: majorisation}

A matrix $A$ is \emph{doubly stochastic} if all its entries are non-negative, and each of its rows and columns has sum equal to 1. A $T$-transform is a special type of doubly stochastic matrix $A$ where, for any $x\in\CC^n$, $Ax=tx+(1-t)\sigma(x)$ for some $t\in[0,1]$ and  a transposition $\sigma\in\mathbb S_n$.

\begin{definition}\label{definition: majorisation}
Given $x,y\in\RR^n$, we say that $x$ is \emph{majorised} by $y$  (notation: $x\prec y$) if
\[
\sum_{j=1}^kx_j^\downarrow\leq\sum_{j=1}^ky_j^\downarrow,\ \ k=1,\ldots,n-1; \ \ \mbox{ and }\ \
\sum_{j=1}^nx_j=\sum_{j=1}^ny_j,
\]
where $x_1^\downarrow,\ldots,x_n^\downarrow$ are the entries of $x$ in non-increasing order. Likewise we denote by $x_j^\uparrow$ the entries of $x$ in non-decreasing order (namely, $x_j^\uparrow=x_{n-j+1}^\downarrow$).
\end{definition}
Majorisation is a well-studied and well-understood notion. Among several characterisations of it, let us mention (see \cite{Bhatia1997} for detailed proofs and further material, and  \cite{MOA} for a more comprehensive treatment and applications):
\begin{theorem}\label{theorem: equivalent conditions for majorisation}
Let $x,y\in\RR^n$. The following statements are equivalent:
\begin{enumerate}
\item\label{theorem: equivalent conditions for majorisation:1} $x\prec y$;
\item\label{theorem: equivalent conditions for majorisation:2} \label{condition: arriba y abajo} $\sum_{j=1}^kx_j^\downarrow\leq\sum_{j=1}^ky_j^\downarrow$ and
$\sum_{j=1}^kx_j^\uparrow\geq\sum_{j=1}^ky_j^\uparrow$ for all $k=1,\ldots,n$;
\item $\sum_j|x_j-t|\leq\sum_j|y_j-t|$ for all $t\in\mathbb R$;
\item\label{theorem: equivalent conditions for majorisation:4} $x=Ay$, where $A\in M_n(\RR)$ is doubly stochastic;
\item\label{theorem: equivalent conditions for majorisation:5} $x=T_r\cdots T_1y$ for some $T$-transforms $T_1,\ldots,T_r$;
\item $x\in\mbox{conv}\{Sy:\ S\in\mathbb S_n\}$;
\item $\sum_j\phi(x_j)\leq\sum_j\phi(y_j)$ for all convex functions $\phi:\mathbb R\to\mathbb R$.
\end{enumerate}
\end{theorem}
\begin{proof}
Since we will only use conditions \eqref{theorem: equivalent conditions for majorisation:1}, \eqref{theorem: equivalent conditions for majorisation:2}, \eqref{theorem: equivalent conditions for majorisation:4}, and \eqref{theorem: equivalent conditions for majorisation:5}, we only sketch the proofs of the implications among  those.

\eqref{theorem: equivalent conditions for majorisation:1}$\iff$\eqref{theorem: equivalent conditions for majorisation:2} follows straightforwardly from the fact that $x_j^\uparrow=x_{n-j+1}^\downarrow$.

\eqref{theorem: equivalent conditions for majorisation:5}$\implies$\eqref{theorem: equivalent conditions for majorisation:4} is also easy since $T$-transforms are doubly stochastic, and the product of doubly stochastic is again doubly stochastic.

\eqref{theorem: equivalent conditions for majorisation:4}$\implies$\eqref{theorem: equivalent conditions for majorisation:1} Let $x=Ay$. We can assume, without loss of generality, that both $x$ and $y$ are ordered non-increasingly---this, from the fact that majorisation does not depend on the order of the entries. We have, for any $k=1,\ldots n$,
\[
\sum_{j=1}^kx_j=\sum_{j=1}^k\sum_{h=1}^nA_{jh}y_h.
\]
 Let $s_i=\sum^k_{j=1}A_{ji}$. Then  \ $0 \leq s_i \leq 1$, $\sum^n_{i=1}s_i=k$ and
$ \sum^k_{j=1}x_j =\sum^n_{i=1}s_iy_i.$
Using that $\sum^n_{i=1}s_iy_k=ky_k$, one can verify that
\begin{equation*}
 \sum^k_{j=1}x_j -\sum^k_{j=1}y_j
=\sum^k_{j=1}(s_j-1)(y_j-y_k)+\sum^n_{j=k+1}(y_j-y_k)s_j\leq0.
\end{equation*}
So \ $\sum^k_{j=1}x_j\leq \sum^k_{j=1}y_j$ \ for all $k$. When $k=n$, the equality is easy to check.

\eqref{theorem: equivalent conditions for majorisation:1}$\implies$\eqref{theorem: equivalent conditions for majorisation:5} Since $y_n\leq x_1\leq y_1$, we can write $x_1$ as a convex combination of $y_1$ and some $y_k$ with $k=\min\{j:\ y_j\leq x_1\}$; this can be implemented by a $T$- transform $T_1$. Now $T_1y$ has first coordinate $x_1$, and it is possible to show that $(x_2,\ldots,x_n)\prec y'$, where $y'$ is $T_1y$ with the first coordinate removed. We can then proceed inductively.
\end{proof}

\bigskip

We mention below three elementary results on majorisation that will be of use later.

\begin{lemma}\label{lemma: 3-dimensional majorisation 2}
Let $x_1,\ldots,x_n\in[0,1]$ with $x_1+\cdots+x_n=k\in\NN$. Then
\[
(x_1,\ldots,x_n)\prec(\overbrace{1,\ldots,1}^{k\text{ times}},\overbrace{0,\ldots,0}^{n-k\text{ times}}).
\]
\end{lemma}

\begin{lemma}\label{lemma: 3-dimensional majorisation}
Let $x_1,\ldots,x_n,\delta\in[0,1]$, $k\in\{1,\ldots,n-1\}$ with
$x_1+\cdots+x_n=\delta+k$. Then
\[
(x_1,\ldots,x_n)\prec(\overbrace{1,\ldots,1}^{k\text{
times}},\ \delta, \overbrace{0,\ldots,0}^{n-1-k\text{ times}}).
\]
\end{lemma}

\begin{lemma}\label{lemma: majorisation concentrates}
Let $x_1,\ldots,x_n,x_1',\ldots,x_n',y_1,\ldots,y_m,y_1',\ldots,y_m'\in\RR$ such that
$x_j'\geq x_j$ and $y_j\geq y_j'$ for all $j$, 
\[
\min\{x_1,\ldots,x_n\}\geq \max\{y_1,\ldots y_m\},
\]
and
\[
x_1'+\cdots+x_n'+y_1'+\cdots+y_m'=x_1+\cdots+x_n+y_1+\cdots+y_m.
\]
Then
\[
(x_1,\ldots,x_n,y_1,\ldots,y_m)\prec(x_1',\ldots,x_n',y_1',\ldots,y_m').
\]
\end{lemma}

\section{The Schur-Horn Theorem}

In this section we offer a proof of the Schur-Horn Theorem \cite{Schur1923,Horn1954} based on the idea of
\cite[Theorem 6]{kadison2002a}; a proof along the same idea appears in \cite{arveson-kadison2006}.
Kadison's trick (Lemma \ref{lemma: the 2x2 trick}) makes the proof very straightforward (modulo Theorem \ref{theorem: equivalent conditions for majorisation}), at the cost of using complex unitaries instead of orthogonal ones as in Horn's original result.

Notation: for $x\in \RR^n$, $D_x\in M_n(\RR)$ is the matrix with diagonal $x$ and zeroes elsewhere. If $A\in M_n(\CC)$, we write $\lambda(A)\in\CC^n$ for the vector whose entries are the eigenvalues of $A$, counting multiplicities.

The next Lemma is Kadison's key idea that allows for a proof of the Schur-Horn Theorem.

\begin{lemma}[Kadison's Trick]\label{lemma: the 2x2 trick}
Let $A\in M_2(\CC)$ be selfadjoint, and let $t\in[0,1]$. Then there exists a unitary $U\in M_2(\CC)$ such that
the diagonal of $UAU^*$ is $tA_{11}+(1-t)A_{22}, (1-t)A_{11}+tA_{22}$.
\end{lemma}
\begin{proof}
Let $\theta$ such that $t=\sin\theta$, and let $c\in\TT$ such
that $cA_{12}=-\overline{c}A_{21}$ (such $c$ always exists since $A_{21}=\overline{A_{12}}$). Then take
\[
U=\begin{bmatrix}c\sin\theta&-\cos\theta\\ c\cos\theta& \sin\theta\end{bmatrix}.\qedhere
\]
\end{proof}

\begin{lemma}\label{lemma: implementation of T-transforms}
Let $y\in\RR^n$, $A\in M_n(\CC)$ selfadjoint with diagonal $y$,
and consider a $T$-transform $T$. Then there exists a unitary $V\in M_n(\CC)$ such that
$VAV^*$ has diagonal $Ty$.
\end{lemma}
\begin{proof}
By definition there exist $t\in[0,1]$ and a transposition $\sigma=(k\ j)$ such that
$Ty=ty+(1-t)\sigma(y)$. So
\[
Ty=(y_1,\ldots,ty_j+(1-t)y_k,\ldots,(1-t)y_j+ty_k,\ldots,y_n).
\]
Applying Lemma
\ref{lemma: the 2x2 trick} to the matrix $A_0=\begin{bmatrix}y_j&A_{jk} \\ A_{kj}& y_k\end{bmatrix}$we get a unitary $U$ such that the diagonal of $UA_0U^*$ consists of $ty_j+(1-t)y_k$ and
$(1-t)y_j+ty_k$.

Let $V$ be the unitary matrix that consists of $U$ in the submatrix corresponding to rows
$j$ and $k$, and the identity everywhere else.
Then $VAV^*$ has diagonal $Ty$.
\end{proof}

Theorem \ref{theorem: Schur-Horn} below is the celebrated Schur-Horn Theorem. 
The implication \eqref{theorem: Schur-Horn:2}$\implies$\eqref{theorem: Schur-Horn:1} is a result due to I. Schur \cite{Schur1923}. The converse is due to A. Horn \cite{Horn1954}. Horn's argument produces unitaries in $M_n(\RR)$; here we opt for Kadison's version which produces unitaries in $M_n(\CC)$, but it implies Corollary \ref{corollary: dispersa delta}, which will be of use in Section \ref{section: main theorems}.
Several proofs of the Schur-Horn Theorem are known; among others, we mention \cite{CarlenLieb2009,ChanLi1983,DillonHeathSustikTropp2005,LeiteRicheTomei1999,Mirsky1958}

\begin{theorem}[Schur-Horn]\label{theorem: Schur-Horn}
Let $x,y\in\RR^n$. Then the following conditions are equivalent:
\begin{enumerate}
\item\label{theorem: Schur-Horn:2} there exists a selfadjoint matrix $A\in M_n(\CC)$ with diagonal $x$ and eigenvalue vector $y$;
\item\label{theorem: Schur-Horn:1} $x\prec y$.
\end{enumerate}
\end{theorem}
\begin{proof}
\eqref{theorem: Schur-Horn:2}$\implies$\eqref{theorem: Schur-Horn:1} As $A$ is selfadjoint, it is diagonalizable
via a unitary: that is, there exists a unitary $U\in M_n(\CC)$ with $UAU^*=D_y$. So $A=U^*D_yU$;
writing this equation for each element of the diagonal of $A$, we get
in components
\[
x_j=(U^*D_yU)_{jj}=\sum_{k,h=1}^n (D_y)_{kh} U_{jk}\overline{U_{jh}}=\sum_{k=1}^n|U_{jk}|^2y_k=(By)_j,
\]
i.e. $x=By$, where $B$ is the matrix with entries $B_{jk}=|U_{jk}|^2$---such matrices are called \emph{orthostochastic} or \emph{unistochastic} in the literature. A straightforward calculation (using $U^*U=UU^*=I$) shows that $B$
is doubly stochastic, and so $x\prec y$ by Theorem \ref{theorem: equivalent conditions for majorisation}.

\eqref{theorem: Schur-Horn:1}$\implies$\eqref{theorem: Schur-Horn:2}
Assume now that $x\prec y$. By Theorem \ref{theorem: equivalent conditions for majorisation}, there
exist $T$-transforms $T_1,\ldots,T_r$ with $x=T_r\cdots T_1y$.

Applying Lemma \ref{lemma: implementation of T-transforms} repeteadly, after $r$ times we get unitaries $V_1,\ldots,V_r$
such that $V_1D_yV_1^*$ has diagonal $T_1y$, $V_2V_1D_yV_1^*V_2^*$ has diagonal $T_2T_1y$, and so on until
$A=(V_r\cdots V_1)D_y(V_r\cdots V_1)^*$ has diagonal $T_r\cdots T_1y=x$. As $A$ is unitarily equivalent with $D_y$, it has eigenvalues  $y$.
\end{proof}

Below we state the Carpenter's Theorem, which is a particular case of the Schur-Horn Theorem above---namely, the case where all entries of $y$ are $0$ and $1$. The name comes from Kadison's picture  in \cite{kadison2002a,kadison2002b} where he sees it as a converse of a generalized Pythagorean Theorem. In turn, the converse of the Pythagorean Theorem rightly deserves that name, since it guarantees that if the sides of a triangle are a Pythagorean triple, then the triangle is a right one---a very useful fact in carpentry.

\begin{corollary}[Carpenter's Theorem]\label{corollary: finite-dimensional carpenter}
Let $a_1,\ldots,a_n\in[0,1]$. Then the following conditions are equivalent:
\begin{enumerate}
\item there exists a projection $P\in M_n(\RR)$ with diagonal $a_1,\ldots,a_n$.
\item $\sum_{j=1}^na_n\in\NN$;
\end{enumerate}
\end{corollary}
\begin{proof}
A projection is a selfadjoint matrix such that all its eigenvalues are zeroes and ones. In such case, Lemma \ref{lemma: 3-dimensional majorisation} guarantees that $\sum a_n=m\in\NN$ is the same as $\{a_1,\ldots,a_n\}\prec(\underbrace{1,\ldots,1}_{m},0,\ldots,0)$.
\end{proof}

\bigskip

The use of Kadison's trick (Lemma \ref{lemma: the 2x2 trick}) to prove the Schur-Horn Theorem allows us to obtain a variation that we will need later. If the matrix $A$ in Corollary \ref{corollary: dispersa delta} is diagonal, then the statement becomes precisely that of the \eqref{theorem: Schur-Horn:1}$\implies$\eqref{theorem: Schur-Horn:2} in the Schur-Horn Theorem \ref{theorem: Schur-Horn}. The converse does not hold, as can be easily seen by taking a non-diagonal $2\times2$ projection. 

\begin{corollary}\label{corollary: dispersa delta}
Let $x,y\in\RR^n$ with $x\prec y$ and let $A\in M_n(\CC)$ be a selfadjoint matrix with diagonal $y\in\RR^n$. Then there exists a unitary $V\in M_n(\CC)$ such that $VAV^*$ has diagonal $x$.
\end{corollary}
\begin{proof}
Note that the proof of \eqref{theorem: Schur-Horn:1}$\implies$\eqref{theorem: Schur-Horn:2} in Theorem \ref{theorem: Schur-Horn} uses Lemma \ref{lemma: implementation of T-transforms}, which works even if the selfadjoint matrix is not selfadjoint. So we can repeat the argument verbatim with $A$ instead of $D_y$. 
\end{proof}

\section{Infinite-Dimensional Majorisation and Diagonals of Selfadjoint Operators}
\label{section: main theorems}

Notation: we work on $\bh $ with $\h$ a separable Hilbert space (this is not an essential restriction, but it simplifies  notation and arguments). We write $\{E_{kj}\}$ for the canonical matrix units in $\bh$; this implies having a fixed orthonormal basis $\{e_j\}$, that will remain fixed throughout. Given any operator $T\in\bh$, we can consider its ``entries'' (thinking of it as an infinite matrix)
\[
T_{kj}=\langle Te_j,e_k\rangle.
\]

The notion of majorisation---Definition \ref{definition: majorisation}---can be clearly applied to infinite sequences, at least under certain conditions. For instance, it makes sense straighforwardly for real sequences in $\ell^1(\NN)$. But one immediately runs into problems: for example, for sequences in $\ell^1(\NN)$ the notions in Definition \ref{definition: majorisation} and in \eqref{condition: arriba y abajo} in Theorem \ref{theorem: equivalent conditions for majorisation} are not equivalent; notwithstanding the fact the the ordering $x_j^\uparrow$ is not even defined for such a sequence. This last objection is not a big one, and was addressed by Neumann \cite{Neumann1999} by considering the numbers
\[
U_k(x)=\sup\{\sum_Kx_j:\ |K|=k\},\ \ L_k(x)=\inf\{\sum_Kx_j:\ |K|=k\}
\]
instead of the sums in \eqref{condition: arriba y abajo} in Theorem \ref{theorem: equivalent conditions for majorisation}, and using
\[
U_k(x)\leq U_k(y),\ \ \ L_k(x)\geq L_k(y),\ \ \ \ k\in\NN
\]
as the definition of $x\prec y$. This makes sense even for $x,y\in\ell^\infty(\NN)$. So one can ask whether a Schur-Horn theorem can be considered in $\bh$; and this  was Neumann's result \cite{Neumann1999}: for real $y\in\ell^\infty(\NN)$,
\begin{equation}\label{equation: neumann}
\overline{\mbox{diag}\{UD_yU^*:\ U\in\mathcal U(\h)\}}^{\|\cdot\|_\infty}=\{x\in\ell^\infty(\NN):\ x\prec y\},
\end{equation}
where $D_y$ is the diagonal operator (in some orthonormal basis) with diagonal $y$, and $\mbox{diag}(T)$ is the diagonal of $T$ seen as an element of $\ell^\infty(\NN)$.

Starting from similar considerations,  majorisation and the Schur-Horn Theorem have been considered in other infinite-dimensional settings, like semifinite von Neumann algebras, and II$_1$-factors in particular \cite{ArgeramiMassey2007,ArgeramiMassey2008a,ArgeramiMassey2008b,ArgeramiMassey2009,
ArgeramiMassey2013,arveson-kadison2006,DykemaFangHadwinSmith2013,Hiai1987,Hiai1989,Kamei1983,Ravichandran2012}.

While Neumann's result \eqref{equation: neumann} is impressive, it does not address the question of what  the possible diagonals of a selfadjoint operator are. What is remarkable is that it does not answer the question even in the case where the spectrum consists of two points. Such characterisation is the content of Kadison's ``Theorem 15'' \cite{kadison2002b}---Theorem \ref{theorem: teorema 15} below.

It is immediately clear that not every $x\in\ell^\infty(\NN)$ with $x_j\in[0,1]$ for all $j$ can be the diagonal of a projection. Indeed, there is the immediate restriction that projections onto finite-dimensional subspaces have integer trace, i.e. we get the restriction $\sum x_j\in\NN$ whenever this sum is finite. Even for sequences with infinite sum, another restriction arises, from the fact that if $P$ is a projection, so is $P^\perp$: this means that $\{x_j\}$ is the diagonal of a projection if and only if $\{1-x_j\}$ is. So, for example, the sequence $3/4,7/8,15/16,\ldots$ is not the diagonal of a projection, because $(1-3/4)+(1-7/8)+\cdots=1/2\not\in\NN$.

The surprise in Kadison's result comes from the case where both $\sum x_j=\sum(1-x_j)=\infty$ (that is, the case of diagonals of projections with both infinite dimension and codimension). In this case no obvious restriction arises as in the other two cases, but Kadison discovered that there is still an obstruction: if we split the elements in the sequence according to their proximity to the two points in the spectrum (namely, $0$ and $1$ in the case of a projection) in two sequences $\{y_j\}$ and $\{z_j\}$ respectively (that is, $y_j\leq1/2<z_j$), and if $\sum y_j<\infty$ and $\sum (1-z_j)<\infty$, then the difference of these two sums has to be an integer; and, moreover, this is the only possible obstruction. For instance, the sequence
\[
\frac{1}{4},\frac{3}{4},\frac{1}{8},\frac{7}{8},\frac{1}{16},\frac{15}{16},\ldots
\]
is the diagonal of a projection with infinite dimension and co-dimension; but there exists no projection with diagonal
\[
\frac{1}{2},\frac{1}{4},\frac{3}{4},\frac{1}{8},\frac{7}{8},\frac{1}{16},\frac{15}{16},\ldots
\]
Indeed, in the first case the difference of the two sums mentioned above is $0$, while in the second case it is $1/2$---not an integer. But we only know this after going through the the proof of \eqref{theorem: teorema 15:0}$\implies$\eqref{theorem: teorema 15:00} in Theorem \ref{theorem: teorema 15} or the corresponding proofs in \cite{kadison2002b} and \cite{arv2006}. We refer the reader to \cite{arv2006} for an analysis and generalisation of the integer obstruction (see also \cite{BownikJasper2013b,Jasper2013}).

\bigskip

The next lemma
plays a key role in our  proof of \eqref{theorem: teorema 15:1}$\implies$\eqref{theorem: teorema 15:0} in Theorem \ref{theorem: teorema 15}.

\begin{lemma}\label{lemma: chebychev}
Let $\{b_j\}$ be a sequence with $0\leq b_j$ for all $j$, and $\sum_jb_j=\infty$. Fix $\delta>0$. Then there exists $n\in\NN$ and coefficients $t_1,\ldots,t_n\in[0,1]$ with $\sum t_j=1$ and such that
$\delta t_j\leq b_{j}$, $j=1,\ldots,n$.
\end{lemma}
\begin{proof}
Since $\sum b_j=\infty$, there exists $n$ such that $\sum_1^nb_j\geq\delta$. Now let $t_j=b_j/\sum_1^nb_k$.
Then $b_j/t_j=\sum_1^nb_k\geq\delta$.
\end{proof}

In the following lemma we summarize a couple of very basic estimates for projections that we will use in 
an essential way in 
the proof of 
Theorem \ref{theorem: teorema 15}. 

\begin{lemma}\label{lemma: properties of projections}
Let $P\in\bh$ be a projection. Then
\begin{equation}\label{equation: projections}
\sum_s|P_{st}|^2=P_{tt},\ \ \ (1-P_{tt})^2+\sum_{s\ne t}|P_{st}|^2=1-P_{tt}.
\end{equation}
In particular, $|P_{st}|^2\leq\min\{P_{tt},P_{ss},1-P_{tt},1-P_{ss}\}$ for any $s,t$, and
\[
\sum_{s\ne t}|P_{st}|^2\leq\min\{P_{tt},1-P_{tt}\}.
\]
\end{lemma}
\begin{proof}
The equalities in \eqref{equation: projections} are simply the equalities $P^2=P$, $(I-P)^2=I-P$, expressed in terms of the entries of $P$. The first estimate follows from \eqref{equation: projections} and the fact that $P_{st}=\overline{P_{ts}}$ (from $P=P^*$). The second estimate follows directly from \eqref{equation: projections} and the fact that, since $P_{tt},1-P_{tt}\in[0,1]$, we have $0\leq P_{tt}-P_{tt}^2\leq P_{tt}$, $0\leq 1-P_{tt}-(1-P_{tt})^2\leq1-P_{tt}$.
\end{proof}

We include two more elementary lemmas:
\begin{lemma}\label{lemma: subsequence}
Let $\{a_n\}_{j\in\NN}\subset(0,1)$ with $\sum_na_n=\infty$. Then there exists a monotone subsequence $\{b_n\}$ of $\{a_n\}$ (possibly after reordering) with $\sum_nb_n=\infty$. 
\end{lemma}
\begin{proof}
If $0$ is the only accumulation point of $\{a_n\}$, we can take $\{b_n\}$ to be a non-increasing reordering of $\{a_n\}$. Otherwise, let $c$ be a nonzero accumulation point. Then there exists a subsequence  $\{b_n''\}$ of $\{a_n\}$ that converges to $c$. Infinitely many of these will be to one side of $c$: say $\{b_n'\}$, with $b_n'\leq c$ (or $b_n'\geq c$, if it is the other case) and $b_n'\to c$. Now we can take $\{b_n\}$ to be a non-decreasing (resp. non-increasing) reordering of $\{b_n'\}$. The sum is clearly infinite as the terms do not go to zero. 
\end{proof}


\begin{lemma}\label{lemma: diagonal with integer partial sums}
Let $\{d_j\}_{j\in\NN}\subset[0,1]$ and let $\NN=\bigcup_kH_k$ be a partition of $\NN$ by finite sets such that 
\[
s_k:=\sum_{j\in H_k}d_j\in\NN
\]
for all $k$. Then there exists a projection $F\in\bh$ with diagonal $\{d_j\}$. 
\end{lemma}
\begin{proof}
For each $k$, we get from Lemma \ref{lemma: 3-dimensional majorisation 2} that
\[
(d_j:\ j\in H_k)\prec(\overbrace{1,\ldots,1}^{s_k\text{ times}},\overbrace{0,\ldots,0}^{|H_k|-s_k\text{times}});
\]  by Corollary \ref{corollary: finite-dimensional carpenter} there exists a matrix projection $F_k$ with diagonal $\{d_j:\ j\in H_k\}$. Then the block-diagonal operator with blocks $F_k$ has the right diagonal. 
\end{proof}

\begin{definition}\label{definition: kadison's condition}
Let $f=\{a_n\}_{j\in\NN}$ be a sequence, with $a_n\in[0,1]$ for all $n$.
Given the sets
\[
N_0=\{n: a_n\leq\frac12\},\ \ \ \ M_0=\{n: a_n>\frac12\},
\]
we define the numbers (possibly infinite)
\begin{equation}\label{equation: a,b}
a_f=\ds\sum_{n\in N_0}a_n,\ \ \  b_f=\ds\sum_{n\in M_0}1-a_n.
\end{equation}
\end{definition}

So now we are in position to state and prove Kadison's celebrated Carpenter's Theorem \cite[Theorem 15]{kadison2002b}. As mentioned above, the proof we provide is quite different---and shorter, although far from trivial---than the original.
Part of the motivation for finding the new proof was the fact that
we have personally witnessed Kadison jokingly asserting, at GPOTS 2006, that he didn't understand
his own proof. Another part came from actual efforts to help some of our graduate students to get to grips with the proof in \cite{kadison2002b}. Recently, Bownik and Jasper \cite{BownikJasper2013a} have published a different proof of the implication \eqref{theorem: teorema 15:00}$\implies$\eqref{theorem: teorema 15:0} in Theorem \ref{theorem: teorema 15}.

We still use Kadison's trick---Lemma \ref{lemma: the 2x2 trick}---in the  proof \eqref{theorem: teorema 15:1}$\implies$\eqref{theorem: teorema 15:0} in Theorem \ref{theorem: teorema 15} in an essential way, and we also use the finite-dimensional Schur-Horn Theorem \ref{theorem: Schur-Horn}. We  dispense with reordering in the inductive step, after first reordering a particular subsequence. For the finite case our proof is radically different from Kadison's, in that we only require the finite-dimensional Schur-Horn Theorem \ref{theorem: Schur-Horn}---as opposed to Kadison's original proof or Bownik-Jasper's where an infinite-dimensional Carpenter's Theorem  is used \cite[Theorem 13]{kadison2002b}, \cite[Theorem 2.1]{BownikJasper2013a}; this allows us to obtain Kadison's Theorems 13 and 14 in \cite{kadison2002b} as straightforward corollaries. The proof of the integer condition is also different from Kadison's  \cite{kadison2002b} and Arveson's \cite{arv2006}; it does not require approximations, and it was inspired by Effros proof \cite{Effros1989} of the fact that if the difference of two projections is trace-class, then its trace is an integer---see \cite{Simon1994} for further results in this direction, that relate to Arveson's ``index'' approach \cite{arv2006}. A direct application of Effros' result does not seem to be possible here, since the differences of projections that arise  are not necessarily trace-class (they are Hilbert-Schmidt, though).

\begin{theorem}[Kadison's Carpenter's Theorem]\label{theorem: teorema 15}
Let $f=\{a_n\}_{n\in\NN}$ be a sequence with $a_n\in[0,1]$ for all $n$. Then the following statements are equivalent:
\begin{enumerate}
\item\label{theorem: teorema 15:0} There exists a projection $P\in\bh$ with diagonal $f$;
\item\label{theorem: teorema 15:00} One of the following holds:
\begin{enumerate}
\item\label{theorem: teorema 15:1} $a_f+b_f=\infty$;
\item\label{theorem: teorema 15:2} $a_f+b_f<\infty$, and $a_f-b_f\in\ZZ$.
\end{enumerate}
\end{enumerate}
\end{theorem}
\begin{proof}

\eqref{theorem: teorema 15:0}$\implies$\eqref{theorem: teorema 15:00} If $a_f+b_f=\infty$, then we are done. So we have to prove that if $a_f,b_f$ are both finite, then $a_f-b_f\in\ZZ$.
The argument that follows is inspired by Effros' proof of Lemma 4.1 in \cite{Effros1989}.
Let $Q=\sum_{n\in N_0}E_{nn}$. It is easy to check that 
\[
\tr(QPQ)=a_f,\ \ \ \tr(Q^\perp P^\perp Q^\perp)=b_f.
\]
As $QPQ$, $Q^\perp P^\perp Q^\perp$ are positive, the equalities above imply that they
are both trace-class.

We also note that $P-Q^\perp$ is compact; actually, it is Hilbert-Schmidt.
Indeed, using  Lemma \ref{lemma: properties of projections},
\begin{align*}
\tr((P-Q^\perp)^2)&=\sum_k\sum_h|(P-Q^\perp)_{kh}|^2 =\sum_{k\in N_0}\left[P_{kk}^2+\sum_{h\ne k}|P_{hk}|^2\right]
\\ &\ \ \ +\sum_{k\in M_0}\left[(P_{kk}-1)^2+\sum_{h\ne k}|P_{hk}|^2\right]\\
&=\sum_{k\in N_0}P_{kk}+\sum_{k\in M_0}(1-P_{kk})=a_f+b_f<\infty.
\end{align*}

As $(P-Q^\perp)^2$ is compact and positive, we may write
\[
(P-Q^\perp)^2=\sum_k\lambda_kR_k,
\]
where $R_1,R_2,\ldots$ are pairwise orthogonal finite-rank projections with sum $I$, and $\lambda_1>\lambda_2>\cdots$ converges to zero. As each
nonzero $\lambda_k$ is an isolated point in the spectrum of $(P-Q^\perp)^2$, there exist continuous functions $f_1,f_2,
\ldots$ such that $R_k=f_k((P-Q^\perp)^2)$. A direct computation shows that $P(P-Q^\perp)^2=(P-Q^\perp)^2P$, $Q(P-Q^\perp)^2=(P-Q^\perp)^2Q$,
from where we deduce by functional calculus that $PR_k=R_kP$, $QR_k=R_kQ$ for all $k$; in particular,  $PR_k$, $QR_k$ are finite-rank projections
for all $k$.

As both $QPQ$, $Q^\perp P^\perp Q^\perp$ are trace-class,
\begin{align*}
a_f-b_f&=\tr(QPQ)-\tr(Q^\perp P^\perp Q^\perp)=\sum_k\tr(QPQR_k)-\tr(Q^\perp P^\perp Q^\perp R_k)\\
&=\sum_k\tr(PR_kQR_k)-\tr( P^\perp R_k Q^\perp R_k)
=\sum_k\tr(PR_kQR_k -P^\perp R_k Q^\perp R_k)\\
&
=\sum_k\tr(PR_k+QR_k- R_k).
\end{align*}
Note that initially we cannot move the $P$ and $Q$ inside the trace, because while $QPQ$ is trace class, $PQ$ and $Q$ are likely not. But  $QPQR_k
=(QR_k)(PR_k)(QR_k)$  is a product of finite rank projections and then we can perform the manipulations in the equalities above.

Now, since
$PR_k$, $QR_k$, and $R_k$ are finite-rank projections, their traces are integers; so $\tr(PR_k+QR_k- R_k)\in\ZZ$ for all $k$.  We have thus shown that
$a_f-b_f$ can be written as a convergent series where all terms are integers; this forces all but finitely many to be zero, and  $a_f-b_f\in\ZZ$.

\bigskip

\eqref{theorem: teorema 15:1}$\implies$\eqref{theorem: teorema 15:0}
We can assume without loss of generality that $0<a_n<1$ for all $n$. This is because it is trivial to get projections with any number of zeroes and/or ones in the diagonal, namely operators of the form $I\oplus0$.
So if we produce a projection with the nonzero $a_n$ in the diagonal, we can later include the ones and zeroes by adding an adequate direct summand.

As we mentioned before, a projection $P$ with diagonal $\{a_n\}$ exists if and only if a projection with diagonal $\{1-a_n\}$ exists (namely, $P^\perp=I-P$). So we can choose at will to work either with the numbers $\{a_n\}$ or the numbers $\{1-a_n\}$. So without loss of generality, let us assume that $a_f=\infty$---if that was not the case, it is for the numbers $1-a_n$.
We note here that it is enough to produce a projection with the $\{a_n\}$ in any order, since any permutation of the diagonal can be achieved by unitary conjugation.

Using Lemma \ref{lemma: subsequence} we will divide the $\{a_n\}$ in two subsequences $\{b_j\}$ and $\{c_j\}$, where the former is  monotone with $\sum b_j=\infty$, and the latter is the rest.  If $\{b_n\}$ is non-decreasing, we make it non-increasing by working with the $\{1-a_n\}$ instead. 

The proof consists of an inductive procedure. In an attempt to help clarity and avoid an abuse of complicated indices, we will show how the procedure works without writing the general induction step. 

The goal is to form a diagonal suitable for Lemma \ref{lemma: diagonal with integer partial sums}, and later tweak it to get the right one. Let $\delta_1=1-b_1$. The first entry in our diagonal will be $b_{m_1}+\delta_1$, with $m_1=1$. We will need to offset this $\delta_1$ somewhere else in the diagonal. Applying Lemma \ref{lemma: chebychev} to $\{b_2,b_3,\ldots\}$ and $\delta_1$, there exist convex coefficients $t_2,\ldots,t_{n_2}$ such that $b_j-t_j\delta_1\geq0$ for all $j=2,\ldots,n_2$. Choose $m_2$ such that $\sum_{n_2+1}^{m_2}b_j\geq1/(1-b_1)$. Let $s_j=b_j/\sum_{n_2+1}^{m_2}b_j$ and choose $\delta_2$ the least positive real number such that 
\[
\delta_2+c_1+\sum_{j=2}^{n_2}(b_j-t_j\delta_1)+\sum_{j=n_2+1}^{m_2}b_j\in\NN.
\]
As we do not want $\delta_2$ in our final diagonal, instead of putting it in an entry of its own we will distribute it among $b_{n_2+1},\ldots,b_{m_2}$. Note that $s_{n_2+1},\ldots,s_{m_2}$ are convex coefficients, and that 
\[
0\leq b_j+s_j\delta_2\leq b_1+s_j\leq b_1+b_j(1-b_1)\leq b_1+1-b_1=1.
\]
So
\begin{equation}\label{equation: diagonal 2}
b_{m_1+1}-t_{m_1+1}\delta_1,\ldots,b_{n_2}-t_{n_2}\delta_1,c_1,b_{n_2+1}+s_{n_2+1}\delta_2,\ldots,
b_{m_2}+s_{m_2}\delta_2
\end{equation}
is a set of numbers in $[0,1]$ with integer sum. Now we repeat the process and we will get
\begin{equation}\label{equation: diagonal 3}
b_{m_2+1}-t_{m_2+1}\delta_2,\ldots,b_{n_3}-t_{n_3}\delta_2,c_2,b_{n_3+1}+s_{n_3+1}\delta_3,\ldots,
b_{m_3}+s_{m_3}\delta_3,
\end{equation}
and so on. We are in position to Apply Lemma \ref{lemma: diagonal with integer partial sums} to obtain a projection $F$ with diagonal $b_1+\delta_1$,\eqref{equation: diagonal 2},\eqref{equation: diagonal 3}, etc.

Next, in the diagonal of $F$, we consider the groups of numbers
\begin{align}\label{equation: diagonal of F}
\nonumber b_{n_k+1}+s_{n_k+1}\delta_k,&\ldots,b_{m_k}+s_{m_k}\delta_k,\\ &b_{m_k+1}-t_{m_k+1}\delta_k,\ldots,
b_{n_{k+1}}-t_{m_{k+1}}\delta_k.
\end{align}
By Lemma \ref{lemma: majorisation concentrates}, the numbers in \eqref{equation: diagonal of F} majorise
$b_{n_k+1},\ldots,b_{m_k},b_{m_k+1},\ldots,b_{n_{k+1}}$.
By Corollary \ref{corollary: dispersa delta}, there exists a unitary $V_k$ that conjugates a selfadjoint matrix with diagonal \eqref{equation: diagonal of F} into one with diagonal 
$b_{n_k+1},\ldots,b_{m_k},b_{m_k+1},\ldots,b_{n_{k+1}}$.

Finally let $V\in\bh$ be the block-diagonal unitary that has $V_k$ in the entries corresponding to the numbers \eqref{equation: diagonal of F} and $1$ everywhere else in the diagonal (i.e. in the entries corresponding to the $c_k$). Then the projection $P=VFV^*$ has diagonal $\{b_j\}\cup\{c_j\}=\{a_n\}$ (in some order) as desired.

\bigskip

\eqref{theorem: teorema 15:2}$\implies$\eqref{theorem: teorema 15:0}

By the convergence of the two series in \eqref{equation: a,b}, we can find finite subsets $N_1\subset N_0$, $M_1\subset M_0$
such that
\[
\delta_1=\sum_{N_0\setminus N_1}a_n<1/2, \ \ \mu_1=\sum_{M_0\setminus M_1}(1-a_n)<\delta_1.
\]

By hypothesis,
\begin{align*}
0< q:&=\ds\sum_{N_1\cup M_1}a_n  + \delta_1-\mu_1
=\sum_{N_1}a_n + \delta_1 - (\sum_{M_1}(1-a_n) + \mu_1)+|M_1| \\
&=a_f-b_f+|M_1|\in\ZZ.
\end{align*}
As $0\leq\delta_1-\mu_1<1$ and $0\leq a_n\leq1$, we get from Lemma \ref{lemma: 3-dimensional majorisation 2} the majorisation
\[
(a_j:\ j\in N_1\cup M_1,\delta_1-\mu_1)\prec(\overbrace{1,\ldots,1}^{q\text{ times}},{0,\ldots,0})
\]

By Corollary \ref{corollary: finite-dimensional carpenter}, there exists a projection $P_0\in \bh$
with diagonal starting with $\{a_j:\ j\in N_1\cup M_1\}\cup\{\delta_1-\mu_1\}$, and zeroes elsewhere.

Next we choose finite subsets $N_2\subset N_0\setminus N_1$, $M_2\subset M_0\setminus M_1$ with
\[
\delta_2:=\sum_{j\in N_0\setminus(N_1\cup N_2)}\,a_j<\frac14,\ \ \mu_2:=\sum_{j\in M_0\setminus(M_1\cup M_2)}\,(1-a_j)<\delta_2.
\]
Then
\begin{align*}
\sum_{N_2}a_j+\sum_{M_2}a_j+\delta_2-\mu_2=\delta_1-\mu_1+|M_2|,
\end{align*}
and Lemma \ref{lemma: 3-dimensional majorisation} implies
\begin{equation}\label{equation: segunda mayorizacion}
(a_j:\ j\in N_2\cup M_2,\delta_2-\mu_2)\prec (\delta_1-\mu_1,\overbrace{1,\ldots,1}^{|M_2|\text{ times}}).
\end{equation}
So Corollary \ref{corollary: finite-dimensional carpenter} guarantees that we can find a unitary that will conjugate a matrix with diagonal the right-hand-side of \eqref{equation: segunda mayorizacion} into one with the left-hand-side of \eqref{equation: segunda mayorizacion} in the diagonal. Then we can use this unitary to conjugate the operator $P_0'$, obtained by replacing $|M_2|$ zeroes with ones in the diagonal of $P_0$, into an operator $P_1$ with diagonal
\[
\{a_j:\ j\in (N_1\cup N_2)\cup(M_1\cup M_2)\}\cup\{\delta_2-\mu_2\}
\]
and zeroes. 
From the fact that $P_0$ is a projection, we deduce that so are $P_0'$ and $P_1$. Now this process can be continued inductively to obtain projections $P_k$ with diagonal starting with \[
\{a_j:\ j\in(N_1\cup\cdots\cup N_k)\cup(M_1\cup\cdots\cup M_k\}\cup\{\delta_{k+1}-\mu_{k+1}\}
\]
and continuing with zeroes, and 
where $0\leq\delta_{k+1}-\mu_{k+1}<2^{-k}$ and
\[
\sum_{s\in N_0\setminus(N_1\cup\cdots\cup N_k)}a_s<2^{-k},\ \ \sum_{s\in M_0\setminus(M_1\cup\cdots\cup M_k)}(1-a_s)<\delta_k<2^{-k}.
\]

We want to show that the sequence $\{P_k\}_k$ converges strongly, as this will imply that its limit is a projection. By construction, this sequence leaves untouched the ``upper left corner'', and thus in the limit the diagonal will contain all the $a_n$.

Fix $e_t$ in the canonical basis. We will prove that $\lim_{k\to\infty}\|(P_{k+r}-P_k)e_t\|=0$ with the rate of convergence not depending on $r$. Since the sequence $\{P_k\}$ is uniformly bounded in norm, this is enough to guarantee strong convergence of the sequence $\{P_k\}$.

Write $\ell_k=1+\sum_1^k|N_k|+|M_k|$; this is the number of non-zero entries in the diagonal of $P_k$. Then
$P_{k+r}$ and $P_k$ agree on the upper left $(\ell_k-1)\times(\ell_k-1)$ block; and recall that the entries of $P_k$ are zero on every row and column beyond the $\ell_k$, and that
$P_{\ell_k,\ell_k}=\delta_{k+1}-\mu_{k+1}$. Also, every diagonal entry of $P_{k+r}-P_k$ is either less than $2^{-k}$ or bigger than $1-2^{-k}$ (all other diagonal entries appear in both $P_{k+r}$ and $P_k$, and get cancelled). 


For $k$ big enough, we will have $t<\ell_k$. This means that on the first $\ell_k-1$ rows, the $t$-column of $P_{k+r}-P_k$ is zero. Then
\begin{align*}
\|(P_{k+r}-P_k)e_t\|^2&=\sum_s|(P_{k+r})_{st}-(P_k)_{st}|^2=\sum_{s\geq\ell_k}|(P_{k+r})_{st}-(P_k)_{st}|^2\\
&\leq2\sum_{s\geq\ell_k}|(P_{k+r})_{st}|^2+2|(P_k)_{\ell_kt}|^2\\
&\ \ \ \ \ \ \text{(using Lemma \ref{lemma: properties of projections})}\\
&\leq2\sum_{s\geq\ell_k}\min\{(P_{k+r})_{ss},1-(P_{k+r})_{ss}\}+2(\delta_{k+1}-\mu_{k+1})\\
&\leq2\sum_{s\in N_0\setminus(N_1\cup\cdots\cup N_k)}a_s\ +\ 2\sum_{s\in M_0\setminus(M_1\cup\cdots\cup M_k)}(1-a_s)+\frac2{2^{k}}\\
&\leq \frac2{2^k}+\frac2{2^k}+\frac2{2^k}=\frac6{2^k}
\end{align*}
for all $r$. So there is a strong limit $P=\lim_kP_k$. Being  a strong limit of projections, $P$ is a projection. Regarding its diagonal, if we fix an index $t$, then
\[
P_{tt}=\langle Pe_t,e_t\rangle=\lim_k\langle P_ke_t,e_t\rangle=\lim_k(P_k)_{tt}=(P_j)_{tt}
\]
for any $j\geq t$. So $P$ has the desired diagonal.
\end{proof}

\begin{remark}\label{remark: the 1/2 does not matter}
We mention here the well-known fact that the $1/2$ used to define $N_0$ and $M_0$ in Definition \ref{definition: kadison's condition} does not play any particular
role: it can be replaced by any other $\alpha\in(0,1)$. Indeed,  for
any sequence of numbers in $[0,1]$ such that $\ds\sum_{a_n\leq\delta}a_n<\infty$,
$\ds\sum_{a_n>\delta}1-a_n<\infty$ for some $\delta>0$, and for any $\alpha,\beta\in(0,1)$,
\[
\sum_{a_n\leq\alpha}a_n-\sum_{a_a>\alpha}1-a_n\in\ZZ \iff
\sum_{a_n\leq\beta}a_n-\sum_{a_a>\beta}1-a_n\in\ZZ.
\]
In Kadison's original proof \cite{kadison2002b}, however, the number $1/2$ does play a key role in his estimates. The proof we presented above works the same if we replace $1/2$ with any  $\alpha\in(0,1)$. So does Bownik and Jasper's \cite{BownikJasper2013a}.
\end{remark}

\bigskip

We note below that Kadison's theorems 13 and 14 in \cite{kadison2002b} can be obtained as straightforward corollaries
of Theorem \ref{theorem: teorema 15}. This was not possible in Kadison's original work \cite{kadison2002b} as his Theorem 13 was used in his proof of his Theorem 15. The same happens in Bownik and Jasper's proof \cite{BownikJasper2013a}. Corollary \ref{theorem: kadison 13} can be seen an infinite-dimensional generalisation of Corollary
\ref{corollary: finite-dimensional carpenter}.

\begin{corollary}[Theorem 13 in \cite{kadison2002b}]\label{theorem: kadison 13}
Let $t_1,t_2,\ldots\ \in[0,1]$. Then the following statements are equivalent:
\begin{enumerate}
\item\label{theorem: kadison 13:1}
there exists a projection $P\in\bh$ with diagonal $t_1,t_2,\ldots$ and trace $m$;
\item\label{theorem: kadison 13:2} $\sum_jt_j=m$.
\end{enumerate}
\end{corollary}
\begin{proof}
\eqref{theorem: kadison 13:1}$\implies$\eqref{theorem: kadison 13:2}: we have that $P$ is trace-class
(as it is of finite-rank), so $\sum_{j=1}^\infty t_j=\text{Tr}(P)=m$.

\eqref{theorem: kadison 13:2}$\implies$\eqref{theorem: kadison 13:1}: the fact that $\sum_jt_j<\infty$
guarantees that $0$ is the only accumulation point of the sequence $\{t_j\}$. Then the set $M_0$, as in Definition
\ref{definition: kadison's condition}, is finite. So
\[
\sum_{j\in N_0}t_j-\sum_{j\in M_0}(1-t_j)=\sum_jt_j-|M_0|=m-|M_0|\in\ZZ
\]
and both sums are finite. By Theorem \ref{theorem: teorema 15} there exists a projection $P$
with diagonal $t_1,t_2,\ldots$
\end{proof}

\begin{corollary}[Theorem 14 in \cite{kadison2002b}]\label{theorem: kadison 14}
Let  $t_1,t_2,\ldots\ \in[0,1]$
Then the following statements are equivalent:
\begin{enumerate}
\item\label{theorem: kadison 14:1}
there exists a projection $P\in\bh$ with diagonal $t_1,t_2,\ldots$ and trace of $P^\perp$ equal to $m$;
\item\label{theorem: kadison 14:2} $\sum_j1-t_j=m$.
\end{enumerate}
\end{corollary}
\begin{proof}
Apply Corollary \ref{theorem: kadison 13} to the numbers $\{1-t_j\}$, and use that  a projection $P$ has diagonal $\{1-t_j\}$ if and only if $P^\perp$ has diagonal $\{t_j\}$.
\end{proof}


\bibliographystyle{abbrv}
\bibliography{martin-refs,sh_bib}

\end{document}